\documentclass[11pt]{amsart}

\usepackage{
	amsmath,
	amsthm,
	amssymb,
	tikz,
	mathdots,
	fancyhdr,
	enumerate,
	setspace,
	multirow,
}

\usetikzlibrary{arrows,calc,matrix,shapes}

\usepackage[
	colorlinks=true,
	linkcolor=black,
	anchorcolor=black,
	citecolor=black,
	menucolor=black,
	filecolor=black,
	urlcolor=black
]
{hyperref}

\usepackage[center,small,sc]{caption}

\usepackage{multicol}

\theoremstyle{plain}
\newtheorem{thm}{Theorem}[section]
\newtheorem{lemma}[thm]{Lemma}
\newtheorem{prop}[thm]{Proposition}
\newtheorem{cor}[thm]{Corollary}

\theoremstyle{definition}
\newtheorem{dfn}[thm]{Definition}
\newtheorem{ex}[thm]{Example}

\newtheorem{remark}[thm]{Remark}
\numberwithin{equation}{section}
\numberwithin{figure}{section}
\numberwithin{table}{section}

\newcommand{\la}{\langle}
\newcommand{\ra}{\rangle}

\newcommand{\bw}{\boldsymbol{w}}
\newcommand{\coll}[2]{{\bf c}^{#1}_{#2}}
\newcommand{\pw}{\bw_0}
\newcommand{\sw}{\bw'}

\newcommand{\bv}{\boldsymbol v}

\newcommand{\PP}{\mathcal{P}}

\tikzset{tab/.style={matrix of math nodes,column sep=-.4, row sep=-.4,text height=8pt,text width=8pt,align=center}}

\begin{document}

\title [Fully commutative elements of type $D$]
{Fully commutative elements of type $D$ and\\ homogeneous representations of KLR-algebras}
\author[G. Feinberg]{Gabriel Feinberg}
\address{Department of Mathematics and Statistics,
Haverford College, Haverford, PA 19041, U.S.A. }
\email{gfeinberg@haverford.edu}
\author[K.-H. Lee]{Kyu-Hwan Lee}
\address{Department of
Mathematics, University of Connecticut, Storrs, CT 06269, U.S.A.}
\email{khlee@math.uconn.edu}

\begin{abstract}
In this paper, we decompose the set of fully commutative elements into natural subsets when the Coxeter group is of type $D_n$, and study combinatorics of these subsets, revealing  hidden structures. (We do not consider type $A_n$ first, since a  similar decomposition for type $A_n$ is trivial.) As an application, we classify and enumerate the homogeneous representations of the Khovanov--Lauda--Rouquier algebras of type $D_n$. 
\end{abstract}

\maketitle

\section*{Introduction}

An element $w$ of a Coxeter group is said to be {\em fully commutative} if any reduced word for $w$ can be obtained from any other  by  interchanges of adjacent commuting generators.
These elements were first introduced by Fan \cite{Fan1995, Fan1996} and Graham \cite{Graham} in their study of  the generalized Temperley--Lieb algebras, where they showed that the generalized Temperley--Lieb algebras have a linear basis indexed by the fully commutative elements. Soon after, these elements were extensively studied by
Stembridge in a series of papers \cite{Stembridge1996,Stembridge1997, Stembridge1998}. 
 He gave a classification of the Coxeter groups
having a finite number of fully commutative elements, which was  previously done
by Graham and Fan in the simply-laced case, and enumerated fully commutative elements for each of the finite types. 
The set of fully commutative elements  also has  connections to Kazhdan--Lusztig cells \cite{Fan1997, FanGreen, GreenLoson}.

In this paper,  we study fully commutative elements of the Coxeter groups of type $D_n$. We decompose the set of fully commutative elements into natural subsets, and study combinatorial properties of these subsets. More precisely, the subsets obtained from the decomposition of the set of fully commutative elements are called {\em packets}.
The main result (Theorem \ref{main D}) shows that each set in the $(n,k)$-packet has size equal to $C(n,k)$, the $(n,k)$-entry of Catalan's triangle, and implies the following identity (Corollary \ref{cor-end}):
\begin{equation} \label{id} \sum_{k=0}^n C(n,k) \left | \PP(n, k) \right | = \frac {n+3} 2 C_n -1 ,\end{equation} where $C_n$ is the $n^{\mathrm{th}}$ Catalan number and $|\PP(n,k)|$ is the number of elements in the $(n,k)$-packet.

It is quite intriguing that the set of fully commutative elements of type $D_n$ has such a structure. A similar decomposition of the set of fully commutative elements of type $A_n$ would be trivial as there would be only one packet.
Since the right-hand side of \eqref{id} is equal to the dimension of the Temperley--Lieb algebra of type $D_n$, the identity \eqref{id} suggests that there might be a  representation theoretic construction, in which   $C(n,k)$ would correspond to the dimension of a representation and  $|\PP(n,k)|$ to its multiplicity. 

As a related work,  Lejczyk  and Stroppel studied these fully commutative elements of type $D_n$ and the Temperley--Lieb algebra action in their paper \cite{Lejczyk2013}, where they also  gave a  diagrammatical description of the parabolic Kazhdan--Lusztig polynomials.

Actually, this research was motivated by the connection of the fully commutative elements  to the homogeneous representations of  the Khovanov--Lauda--Rouquier (KLR) algebras (also known as quiver Hecke algebras).
Introduced by Khovanov and Lauda \cite{Khov2009} and independently by Rouquier \cite{Rouq2008}, the KLR algebras have been the focus of many recent studies.  In particular, these algebras categorify the lower (or upper) half of a quantum group. 
In the paper \cite{Klesh2010}, Kleshchev and Ram significantly reduce the problem of describing the irreducible representations of the KLR algebras to the study of {\em cuspidal} representations for finite types. In the process of constructing the cuspidal representations, Kleshchev and Ram defined  a class of representations known as {\em homogeneous representations} \cite{Klesh2008}, those that are concentrated in a single degree.   Homogeneous representations  include most of the cuspidal representations for finite types with a suitable choice of ordering on words. In particular, for type $D_n$, a natural choice of ordering makes all the cuspidal representations homogeneous.  Therefore it is important to completely understand homogeneous representations.
The main point in our approach is that  homogeneous representations can be constructed from the sets of reduced words of fully commutative elements in the corresponding Coxeter group as shown in \cite{Klesh2008}.

As an application intended at the beginning of this research, our results classify and enumerate the homogeneous representations of KLR algebras according to the decomposition of the set of fully commutative elements. See Corollary  \ref{cor-DD}. 

The outline of this paper is as follows. In Section \ref{first}, we fix notations, briefly review the representations of KLR algebras,  and explain the relationship between homogeneous representations and fully commutative elements of a Coxeter group. In the next section, we study the set of fully commutative elements of type $D_n$, construct explicit bijections among packets, prove the main theorem (Theorem \ref{main D}) and obtain the identity \eqref{id}.

\subsection*{Acknowledgments}
The authors would like to thank Catharina Stroppel for helpful comments and the referee for many useful comments.
Part of this research was performed	while both authors were visiting Institute for Computational and Experimental Research in Mathematics (ICERM) during the spring of 2013 for the special program ``Automorphic Forms, Combinatorial Representation Theory and Multiple Dirichlet Series".  They wish to thank the organizers and staff. 

\section{Homogeneous Representations and Fully Commutative Elements} \label{first}

\subsection{Definitions}
To define a KLR algebra, we begin with a quiver $\Gamma$.  In this paper, we will focus mainly on quivers of Dynkin type $D_n$, but for the definition, any finite quiver with no double bonds will suffice.  Let $I$ be the set indexing the vertices of $\Gamma$,
and for indices $i\neq j$, we will say that $i$ and $j$ are neighbors if $i\rightarrow j$ or $i \leftarrow j$.  Define $Q_+ = \bigoplus_{i\in I} \mathbb{Z}_{\geq 0} \, \alpha_i$ as the non-negative lattice with basis $\{\alpha_i | i\in I\}$.
The set of all words in the alphabet $I$ is denoted by $\langle  I \rangle$, and for a fixed $\alpha = \sum_{i\in I}c_i\alpha_i \in Q_+$, let $\langle I \ra_\alpha$ be the set of words $\bw$ on the alphabet $I$ such that each $i\in I$ occurs exactly $c_i$ times in $\bw$.  We define the {\em height}   of $\alpha$ to be $\sum_{i\in I}c_i$.  We will write $\bw = [{w_1}, {w_2}, \hdots, {w_d}]$, $w_j \in I$.

Now, fix an arbitrary ground field $\mathbb F$ and choose an element $\alpha\in Q_+$.  Then the \textit{Khovanov--Lauda--Rouquier algebra} $R_\alpha$ is the associative $\mathbb F$-algebra generated by:
 \begin{itemize}
  \item  idempotents
  $\{e(\bw) ~|~ \bw\in \la I \ra_\alpha\}$,

   \item  symmetric generators $\{\psi_1, \hdots, \psi_{d-1}\}$ where $d$ is the height of the root $\alpha$,
  \item  polynomial generators
  $\{y_1, \hdots, y_d\}$,
  \end{itemize}

  subject to relations
  \begin{align} \label{eqn-1}
 & e(\bw)e(\bv) = \delta_{\bw\bv}e(\bw), \quad \quad \sum_{\bw\in \la I \ra_\alpha}e(\bw) = 1;
  \\ & y_ke(\bw) = e(\bw)y_k;
   \\  \label{eqn-2}
& \psi_ke(\bw) = e(s_k\bw)\psi_k;
  \\ 
& y_ky_\ell = y_\ell y_k;
 \\ 
& y_k\psi_\ell =\psi_\ell y_k \ (\hbox{for } k\neq \ell,\ell+1);
  \\  &   (y_{k+1}\psi_k - \psi_ky_k)e(\bw) =
    \left\{ \begin{array}{l l}
     e(\bw) & \hbox{if } w_k = w_{k+1},\\
      0 & \hbox{otherwise; }\\
    \end{array}\right.
   \\    & (\psi_k y_{k+1}- y_k\psi_k)e(\bw) =
    \left\{ \begin{array}{l l}
     e(\bw) & \hbox{if } w_k = w_{k+1},\\
      0 & \hbox{otherwise; }\\
    \end{array}\right.
      \\ 
    & \psi_k^2e(\bw) =
    \left\{ \begin{array}{l l}
    0 & \hbox{if } w_k = w_{k+1},\\
      (y_k-y_{k+1})e(\bw) & \hbox{if } w_k\rightarrow w_{k+1},\\
      (y_{k+1}-y_{k})e(\bw) & \hbox{if } w_k\leftarrow w_{k+1},\\
     e(\bw) & \hbox{otherwise; }\\
    \end{array}\right.
  \\ 
  &  \psi_k\psi_\ell = \psi_\ell\psi_k \ (\hbox{for } |k-\ell|>1);
  \\  &  (\psi_{k+1}\psi_{k}\psi_{k+1} -\psi_{k}\psi_{k+1}\psi_{k})e(\bw) =
    \left\{ \begin{array}{l l}
       e(\bw)& \hbox{if } w_{k+2} = w_k\rightarrow w_{k+1},\\
       -e(\bw)& \hbox{if } w_{k+2} = w_k\leftarrow w_{k+1},\\
     0 & \hbox{otherwise. }\\
    \end{array}\right.
  \end{align}
Here $\delta_{\bw\bv}$ in \eqref{eqn-1} is the Kronecker delta and, in \eqref{eqn-2}, $s_k$ is the $k^\textrm{th}$ simple transposition in the symmetric group $S_d$, acting on the word $\bw$ by swapping the letters in the $k^\textrm{th}$ and $(k+1)^\textrm{st}$ positions.
If $\Gamma$ is a Dynkin-type quiver, we will say that $R_\alpha$ is a KLR algebra of that type.

We impose a $\mathbb Z$-grading on $R_\alpha$ by
\begin{align} \label{eqn-degree}
 & \deg(e(\bw))  =  0, \quad
 \deg(y_i) = 2, \\  & \deg(\psi_i e(\bw))  =   \left\{\begin{array}{rl}
 -2 & \hbox{ if } w_i= w_{i+1},\\
 1 & \hbox{ if } w_i, w_{i+1} \hbox{ are neighbors in } \Gamma ,\\
 0 & \hbox{ if } w_i, w_{i+1} \hbox{ are not neighbors in } \Gamma .\\
\end{array}\right. \label{eqn-degree-1}
\end{align}

Set $R = \bigoplus_{\alpha\in Q_+} R_\alpha$, and let $\text{Rep}(R)$ be the category of finite dimensional graded $R$-modules, and denote its Grothendieck group by $[\text{Rep}(R)]$. Then $\text{Rep}(R)$ categorifies one half of the quantum group. More precisely, let $\mathbf f$ and $'\mathbf f$ be the Lusztig's algebras defined in \cite[Section 1.2]{lusztig2011introduction} attached to the Cartan datum encoded in the quiver $\Gamma$ over the field $\mathbb Q(v)$. We put $q=v^{-1}$ and $\mathcal A=\mathbb Z[q, q^{-1}]$, and let $'\mathbf f_{\mathcal A}$ and $\mathbf f_{\mathcal A}$ be the $\mathcal A$-forms of $'\mathbf f$ and $\mathbf f$, respectively. Consider the graded duals $'\mathbf f^*$ and $\mathbf f^*$, and their $\mathcal A$-forms \['\mathbf f^*_{\mathcal A} := \{ x \in {'\mathbf f}^* : x('\mathbf f_{\mathcal A}) \subset \mathcal A\} \ \text{ and } \ \mathbf f^*_{\mathcal A} := \{ x \in \mathbf f^* : x(\mathbf f_{\mathcal A}) \subset \mathcal A\}.\]
Then we have:
\begin{thm} \cite{Khov2009}
There is an $\mathcal A$-linear (bialgebra) isomorphism $\gamma^*:[\mathrm{Rep}(R)]  \xrightarrow{\sim} \mathbf f^*_{\mathcal A}$.
\end{thm}
Since it is not directly related to our purpose, we omit defining the bialgebra structures. Details  can be found in \cite{Khov2009, Klesh2010}.

\medskip

A word $\mathbf i \in \la I \ra_\alpha$ is naturally considered as an element of $'\mathbf f_{\mathcal A}^*$ to be dual to the corresponding monomial in $'\mathbf f_{\mathcal A}$.
Let $M$ be a finite dimensional graded $R_\alpha$-module. Define the $q$-character of $M$ by
\[ \text{ch}_q \, M := \sum_{\mathbf i \in \la I \ra_\alpha} (\dim_q M_{\mathbf i} )\, \mathbf i \in {'\mathbf f}^*_{\mathcal A} ,\] where $M_{\mathbf i} = e(\mathbf i) M$ and $\dim_q V:= \sum_{n \in \mathbb Z} (\dim V_n) \, q^n \in \mathcal A$ for $V=\oplus_{n \in \mathbb Z} V_n$.
A non-empty word $\mathbf i$ is called {\em Lyndon} if it is lexicographically smaller than all its proper right factors, or equivalently smaller that all its rotations.
For $x \in {'\mathbf f}^*$ we denote by $\max(x)$ the largest word appearing in $x$. A word $\mathbf i \in \langle I \rangle$ is called {\em good} if there is $x \in \mathbf f^*$ such that $\mathbf i = \max(x)$. Given a module $L \in \text{Rep}(R_\alpha)$, we say that $\mathbf i \in \langle I \rangle$ is the {\em highest weight} of $L$ if $\mathbf i = \max(\text{ch}_q \, L)$. An irreducible  module 
$L \in {R_\beta}$ is called {\em cuspidal} if its highest weight is a good Lyndon word. It is known that the set of good Lyndon words is in bijection with the set of positive roots $\Delta_+$ of the root system attached to $\Gamma$, when $\Gamma$ is of finite type.

Using the techniques developed by Leclerc in \cite{Leclerc2004}, Kleshchev and Ram, and then Melvin, Mondragon, and Hill showed:
\begin{thm}[\cite{Klesh2010}; \cite{Hill2012}, 4.1.1] Assume that $\Gamma$ is of finite Dynkin type. Then the good Lyndon words parameterize the cuspidal  representations of the KLR algebra $R_\beta$, $\beta \in \Delta_+$.  In turn, any irreducible graded $R_\alpha$-module  for $\alpha \in Q_+$ is given by an  irreducible head of a standard representation induced from cuspidal representations up to isomorphism and degree shift.
\end{thm}

The above theorem clearly explains the importance of cuspidal representations. In the next subsection, we will introduce another class of representations which contains most of  the cuspidal representations with a suitable choice of ordering on words.

\subsection{Homogeneous representations}
  We define a \textit{homogeneous representation} of a KLR algebra to be an irreducible, graded representation fixed in a single degree (with respect to the $\mathbb Z$-grading described in  \eqref{eqn-degree} and \eqref{eqn-degree-1}).   
Homogeneous representations form an important class of irreducible modules since most of the cuspidal representations are  homogeneous  with a suitable choice of ordering on $\la I \ra$ (\cite{Klesh2010, Hill2012}). After introducing some terminology, we will  describe these representations in a combinatorial way. We continue to assume that $\Gamma$ is a simply-laced quiver.

  Fix an $\alpha\in Q_+$ and let $d$ be the height of $\alpha$. For any word $\bw\in \la I \ra_\alpha$, we say that the simple transposition $s_r\in S_d$ is an \textit{admissible transposition for $\bw$} if the letters $w_r$ and $w_{r+1}$ are neither equal nor neighbors in the quiver $\Gamma$.  Following Kleshchev and Ram \cite{Klesh2008}, we define the \textit{weight graph} $G_\alpha$ with vertices given by $\la I \ra_\alpha$.  Two words $\bw$, $\bv \in \la I \ra_\alpha$ are connected by an edge if there is an admissible transposition $s_r$ such that $s_r\bw = \bv$.

We say that a connected component $C$ of the weight graph $G_\alpha$ is \emph{homogeneous} if the following property holds for every $\bw\in C$:
  \begin{eqnarray} \label{homog}
 &   &  \hbox{If } w_r  =  w_s  \hbox{ for some } 1\leq r<s \leq d \hbox{, then there exist } t,u\\\nonumber
    & & \hspace*{1.8 cm} \hbox{ with } r<t<u<s \hbox{ such that }  w_r \hbox{ is neighbors with both } w_t \hbox{ and } w_u.
  \end{eqnarray}
A word satisfying condition~\eqref{homog} will be called a \textit{homogeneous word}.

  \begin{ex}  
  Consider the $A_3$ quiver
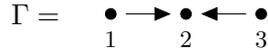
\[
\begin{tikzpicture}[>=triangle 45]
\draw[fill=black, opacity=1, ->] (1,0) circle (2pt)
 (2,0) circle (2pt)
 (3,0) circle (2pt);
\draw
	(0,0) node {$\Gamma =$}
	(1,-.1) node[below]{\footnotesize 1}
	(2,-.1) node[below]{\footnotesize 2}
	(3,-.1) node[below]{\footnotesize 3};
\draw[->] (1.2,0)--(1.8,0);
\draw[->] (2.8,0)--(2.2,0);
\end{tikzpicture}
\]
We then have $I=\{1,2,3\}$, and  choose the element $\alpha = \alpha_1+ 2\alpha_2 + \alpha_3\in Q_+$.  Then the weight graph $G_\alpha$ is given by:

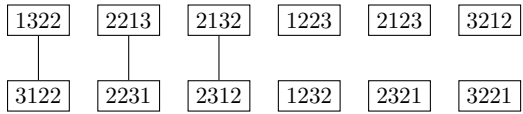
\begin{figure}[h]
 \[
  \begin{tikzpicture}[x=1.5cm, y=1cm, scale=0.8]
  \draw[nodes={draw,rectangle, fill=white ,fill opacity=1, scale=0.8}]
        (1,0) node[above]{1322}--(1,-1) node{3122}
        (2,0) node[above]{2213}--(2,-1) node{2231}
         (3,0) node[above]{2132}--(3,-1) node{2312}
         (4,0) node[above] {1223}
         (4,-1) node {1232}
         (5,0) node[above] {2123}
         (5,-1) node {2321}
         (6,0) node[above] {3212}
         (6,-1) node {3221} ;
  \end{tikzpicture}
  \]
\end{figure}

 One can see that the only homogeneous component is
  
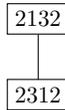
\[
  \begin{tikzpicture}[x=1.5cm, y=1cm, scale=0.8]
  \draw[nodes={draw,rectangle, fill=white ,fill opacity=1, scale=0.8}]
         (3,0) node[above]{2132}--(3,-1) node{2312};
  \end{tikzpicture}
  \]
   In this case, the two instances of the letter $2$ have the neighbors $1$ and $3$ occurring between them.  Note that, for some $\alpha\in Q_+$ we may have that every component of the weight graph is homogeneous (e.g. $\alpha= \alpha_1 + \alpha_2 + \alpha_3$), while for others we may see that no components are homogeneous (e.g. $\alpha=2\alpha_1 + \alpha_2$).
  \end{ex}

 A main theorem of \cite{Klesh2008} shows that the homogeneous components of $G_\alpha$ exactly parameterize the homogeneous representations of the KLR algebra $R_\alpha$:

 \begin{thm}[\cite{Klesh2008}, Theorem 3.4] \label{com} Let $C$ be a homogeneous component of the weight graph $G_\alpha$.  Define an $\mathbb F$-vector space $S(C)$ with basis $\{v_{\bw}~|~\bw\in C \}$ labeled by the vertices in $C$.  Then we have an $R_\alpha$-action on $S(C)$ given by
  \begin{eqnarray*}
  	e(\bw')v_{\bw} & = & \delta_{\bw,\bw'}v_{\bw} \quad (\bw'\in \la I \ra_\alpha, \bw\in C), \\
  	y_rv_{\bw} &=& 0 \quad (1\leq r \leq d, \bw\in C), \\
  	\psi_rv_{\bw} & = &
  	  \left\{\begin{array}{ll}
  	  	v_{s_r\bw} & \textrm{ if } s_r\bw\in C \\
  	  	0 & \textrm{otherwise} 
  	  \end{array}\right.   \quad (1\leq r \leq d-1, \bw\in C), 
  \end{eqnarray*}
  which gives $S(C)$ the structure of a homogeneous, irreducible $R_\alpha$-module.  Further $S(C)\ncong S(C')$ if $C\neq C'$, and this construction gives all of the irreducible homogeneous modules, up to isomorphism.
 \end{thm}

As a result, the task of identifying homogeneous modules of a KLR algebra is reduced to identifying homogeneous components in a weight graph.  This is simplified further by the following lemma:

  \begin{lemma}[\cite{Klesh2008}, Lemma 3.3] A connected component $C$ of the weight graph $G_\alpha$ is homogeneous if and only if an element $\bw\in C$ satisfies the condition~\eqref{homog}.
  \end{lemma}

 Recall that we call a word satisfying condition~\eqref{homog} a {homogeneous word}.  The homogeneous words have other combinatorial characterizations, which we explore in the next subsection.


\subsection{Fully commutative elements of Coxeter groups}

Since the homogeneity of $\bw \in \la I \ra$ does not depend on the  orientation of a quiver,  it is enough to consider Dynkin diagrams and the corresponding Coxeter groups.
Given a simply laced Dynkin diagram,  the corresponding Coxeter group will be denoted by $W$ and the generators by $s_i$, $i \in I$. A reduced expression $s_{i_1} \cdots s_{i_r}$ will be identified with the word $[i_1,  \dots, i_r]$ in  $ \la I \ra$. For example, in the type $A_4$, the reduced expressions 
\[
s_3s_1s_2s_3s_4 =  s_1s_3s_2s_3s_4 = s_1s_2s_3s_2s_4 = s_1s_2s_3s_4s_2
\]
 are identified with the words \[
    [3,1,2,3,4], \  [1,3,2,3,4] ,\   [1,2,3,2,4], \  [1,2,3,4,2] ,\  \text{ respectively}.
 \]
The identity element will be identified with the empty word $[~]$.

Assume that  $W$ is a simply-laced Coxeter group. An element $w \in W$ is said to be {\em fully commutative} if any reduced word for $w$ can be obtained from any other  by  interchanges of adjacent commuting generators, or equivalently if no reduced word for $w$ has $[i, i', i]$ as a subword where $i$ and $i'$ are neighbors in the Dynkin diagram. We make several observations, which are important for our study of homogeneous representations, and list them in the following lemma. These observations were first made by Kleshchev and Ram.

\begin{lemma} \cite{Klesh2008} \label{lem-equiv}
\hfill
\begin{enumerate}
\item A homogeneous component of the weight graph  $G_\alpha$ contains as its vertices exactly the set of reduced expressions for a fully commutative element in $W$. 

\item The set of homogeneous components is in bijection with the set of fully commutative elements in $W$. 

\item  Any KLR algebra of type $A_n~(n\geq 1)$, $D_n~(n\geq 4)$, or $E_n~(n=6,7,8)$ has finitely many irreducible homogeneous representations.

\end{enumerate}

\end{lemma}

\begin{proof}
Part (1) follows from the definitions; in particular, the condition \eqref{homog} implies that no word in a homogeneous component has $[i, i', i]$ as a subword where $i$ and $i'$ are neighbors. Parts (2) and (3) are consequences of (1).
\end{proof}

Stembridge \cite{Stembridge1996} classified all of the Coxeter groups that have finitely many fully commutative elements. The list includes the infinite families of types $E_n$, $F_n$ and $H_n$. His results completed the work of Fan \cite{Fan1996}, who had done this for the simply-laced types.
In the same paper \cite{Fan1996}, Fan showed that the fully commutative elements parameterized natural bases for corresponding quotients of Hecke algebras.  In type $A_n$, these give rise to the Temperley--Lieb algebras (see \cite{Jones1987}).
Fan and Stembridge also enumerated the set of fully commutative elements.  In particular, they showed the following.

\begin{prop}[\cite{Fan1996,Stembridge1998}]\label{full comm count}  Let $C_n$ be the $n^\textrm{th}$ Catalan number, i.e. $C_n = \frac{1}{n+1}{2n\choose n}$.
Then the number of fully commutative elements in the Coxeter group of type $A_n$ is $C_{n+1}$, and that of type  $D_n$ is  $\frac{n+3}{2}C_n-1$.
\end{prop}

We immediately obtain a consequence on homogeneous representations from Lemma \ref{lem-equiv}.

\begin{cor}
 A KLR algebra $R = \bigoplus_{\alpha\in Q_+}R_\alpha$ of type $A_n$ has $C_{n+1}$ irreducible homogeneous representations, while a KLR algebra of type $D_n$ has $\frac{n+3}{2}C_n-1$ irreducible homogeneous representations.
\end{cor}

In \cite{Klesh2008}, Kleshchev and Ram parameterized homogeneous representations using skew shapes. In this paper, we will decompose the set of fully commutative elements to give a finer enumeration of homogeneous representations in type $D_n$. 
More precisely, our main theorem (Theorem \ref{main D}) proves that these homogeneous representations can be organized naturally into {\em packets} (defined in Section \ref{sec-pac}), and counted by Catalan's triangle.
Note that these results contribute not only to the combinatorics of the representation theory of KLR algebras, but also to the study of fully commutative elements of Coxeter groups.


\section{Packets in  Type $D_n$}\label{ch:D}

In this section, we will focus on fully commutative elements and representations of KLR-algebras of type $D_n$.  That is, we shall assume that $\Gamma$ is a quiver whose underlying graph is of the form:
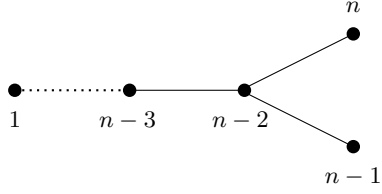
\[
 \begin{tikzpicture}[x=2cm, scale=.5]
    \foreach \x in {0,1.5,3}
    \draw[xshift=\x,thick, fill=black] (\x,0) circle (1.5 mm);
    \draw[thick, fill=black] (4.5,1.5) circle (1.5 mm);
    \draw[thick, fill=black] (4.5,-1.5) circle (1.5 mm);
    \draw (1.5,0)--(3,0);
    \draw[dotted, thick] (0,0) -- +(1.5,0);
    \foreach \y in {1.5}
    \draw (3,0)--(4.5,1.5)
          (3,0)--(4.5,-1.5);
    \draw (0,-.3) node[below]{\footnotesize1};
    \draw (1.5,-.3) node[below]{\footnotesize $n-3$};
    \draw (3,-.3) node[below]{\footnotesize $n-2$};
    \draw (4.5,1.8) node[above]{\footnotesize $n$};
    \draw (4.5,-1.8) node[below]{\footnotesize $n-1$};
  \end{tikzpicture}
\]
 We begin with canonical reduced words of type $D_n$.

\subsection{Canonical reduced words}
 
  For $1\leq i \leq n-1$, we define the words $s_{ij}$by:
\[ s_{ij} =
 \left\{
   \begin{array}{ll}
     ~[i, i-1, \hdots, j] & \text{ if } i \ge j, \\  ~[~] & \textrm{ if } i<j .
   \end{array}
 \right.
\]
  When $i=n$, we define
\[ s_{nj} =
 \left\{
   \begin{array}{ll}
     ~[n, n-2, \hdots, j] & \textrm{ if } j\leq n-2, \\
    ~[n]  & \textrm{ if } j=n,n-1, \\
     ~[~] & \textrm{ if } j>n.
   \end{array}
 \right.\]
We will often write $s_{ii}=s_i$.
The following lemma provides a canonical form we need.

\begin{lemma}[\cite{Bokut2001}, Lemma 5.2]\label{dcan}
Any element of the Coxeter group of type $D_n$ can be uniquely written in the reduced form
 \[
   s_{1i_1}s_{2i_2}\cdots s_{n-1i_{n-1}}s_{nj_1}s_{n-1j_2}s_{nj_3}s_{n-1j_4}\cdots s_{n-1+[\ell]_2\, j_\ell}
 \]
 where $1 \le i_k \leq k+1$ for $1 \le k \le n-1$, and   $1\leq j_1<j_2<\cdots<j_\ell\leq n-1$ for $\ell\geq0$, and $[\ell]_2=1$ when $\ell$ is odd, $[\ell]_2=0$ when $\ell$ is even.
\end{lemma}

The left factor $s_{1i_1}s_{2i_2}\cdots s_{n-1i_{n-1}}$  will be called the \textit{prefix}, and similarly the right factor $s_{nj_1}s_{n-1j_2}s_{nj_3}s_{n-1j_4}\cdots s_{n-1+[\ell]_2\, j_\ell}$ will be called the \textit{suffix} of the reduced word.  For example,  in the case of $D_5$, the word $s_{21}s_3s_4s_{52}s_{43}s_5=[2,1,3,4,5,3,2,4,3,5]$ has prefix $s_{21}s_3s_4=[2,1,3,4]$ and suffix $s_{52}s_{43}s_5=[5,3,2,4,3,5]$.  Given a reduced word $\bw$ in canonical form, we will denote by $\pw$ the prefix of $\bw$ and by $\sw$ the suffix, and write $\bw = \pw \sw$. Generally, a word of the form $s_{nj_1}s_{n-1j_2}s_{nj_3}s_{n-1j_4}\cdots s_{n-1+[\ell]_2\, j_\ell}$ with $1\leq j_1<j_2<\cdots<j_\ell\leq n-1$ for $\ell \ge 0$ will be called a {\em suffix}. 

\begin{remark}
Notice that choosing a suffix is equivalent to choosing a (possibly empty) subset of $\{1, 2, \hdots, n-1\}$.  There are $2^{n-1}$ ways to do this.  Since there are $n!$ prefixes,  we have $n!\cdot 2^{n-1}$ elements in the canonical reduced form. We recall that there are the same number of elements in the type-$D_n$ Coxeter group.
\end{remark}

\begin{lemma} \label{every suffix}
  Every  suffix is a homogeneous word, and so represents a fully commutative element.
\end{lemma}

\begin{proof}
A suffix $\bw'$ cannot have $[1, \dots , 1]$ as a subword. If a suffix $\bw'$ has a subword $[k, \dots , k]$ for $2 \le k \le n-2$, then $\bw'$ must have $[k, k-1, \dots , k+1, k]$ as a subword from the conditions on a suffix, and $\bw'$ is a homogeneous word from the definition \eqref{homog}. If $\bw'$ has $[n-1, \dots , n-1]$ as a subword, then $\bw'$ must have $[n-1, n-2, \dots , n, n-2, \dots , n-1]$. Similarly, if $\bw'$ has $[n, \dots , n]$, then $\bw'$ must have $[n, n-2, \dots , n-1, n-2, \dots , n]$. In both cases, $\bw'$ is a homogeneous word.
\end{proof}

Let $\mathcal W_n$ be the set of canonical reduced words of type $D_n$ given in Lemma \ref{dcan}. Each homogeneous word in $\mathcal W_n$ uniquely represents a fully commutative element of the Coxeter group of type $D_n$ and also a homogeneous component of a weight graph by Lemma \ref{lem-equiv}. The homogeneous words in $\mathcal {W}_n$ will be grouped based on their suffixes:

\begin{dfn}
 A {\em collection} $\coll{n}{\sw} \subset \mathcal {W}_n$ labeled by a suffix $\sw$ is defined to  be the set of homogeneous words in $\mathcal {W}_n$ whose  suffix is $\sw$. 
A collection will be identified with the set of corresponding fully commutative elements in the Coxeter group of type $D_n$.
\end{dfn}

Some of the collections have the same number of elements as we will see in the following lemma and proposition.

\begin{lemma}\label{type2}
 For a fixed $k$, $0\leq k \leq n-3$, any collection labeled by a suffix of the form 
\begin{equation} \label{ff} s_{n\, k+1}s_{n-1j_2}s_{nj_3}s_{n-1j_4}\cdots s_{n-1+[\ell]_2\, j_\ell}\quad (\ell \ge 2) \end{equation} has the same set of prefixes. In particular, these collections have the same number of elements.
\end{lemma}

\begin{proof}

Let $\bw'$ be a suffix of the form \eqref{ff}. Then $\bw'$ has the suffix $\bw_1:=s_{n\, k+1} s_{n-1}$ as a subword. Since removing letters from the end of a word will not affect its homogeneity, it is clear that any prefix appearing in the collection $\coll{n}{\bw'}$ also appears in $\coll{n}{\bw_1}$.  We need to show, then, the opposite inclusion. 

  Suppose now that $\pw$ is a prefix of a homogeneous word appearing in the collection labeled by $\bw_1$. Since the prefix and suffix of a homogeneous word are individually homogeneous words, we only assume  that there is some letter $r$ which appears in both $\pw$ and $\sw$. From the condition $$1\leq j_1<j_2<\cdots<j_\ell\leq n-1$$ on the suffix $\sw$, we see that the letter $r$ also appears in $\bw_1$. The homogeneity of $\bw_0\bw_1$ requires that two neighbors of $r$ appear between the instance of $r$ in $\pw$ and the  instance of $r$ in $\bw_1$, and these two neighbors of $r$ also satisfy the homogeneity condition for $\pw \sw$ since $\bw_1$ is a left factor of $\sw$. This proves that $\pw$ is a prefix of $\sw$ for any suffix $\sw$ of the form \eqref{ff}.
\end{proof}

\begin{prop}\label{type size}
  For $1 \leq k\le n-3$, the collection labeled by the suffix $s_{nk}$ has the same number of elements as any of the collections labeled by the suffix of the form  \[s_{n\, k+1}s_{n-1j_2}s_{nj_3}s_{n-1j_4}\cdots s_{n-1+[\ell]_2\, j_\ell}\quad (\ell \ge 2). \]
\end{prop}

\begin{proof}
Let $\bw_1= s_{n \, k+1}s_{n-1}$ and $\bw_2=s_{nk}$. By Lemma \ref{type2}, it is enough to establish a bijection between the collections $\coll{n}{\bw_1}$ and $\coll{n}{\bw_2}$.
We define a map $\sigma: \coll{n}{\bw_2} \to \coll{n}{\bw_1}$ as follows.
  Suppose that $ \pw$ is the prefix of the word $\bw=\pw \bw_2=\pw [n, n-2, \dots , k] \in \coll{n}{\bw_2}$, and let $r$ be the last letter of $\pw$. Then by the condition of homogeneity, we must have $r<k$ or $r=n-1$.
If $r<k$, we define $\sigma(\bw)=\pw \bw_1$, i.e. the map $\sigma$ will simply replace the suffix $\bw_2=[n, n-2, \dots, k+1, k]$ with the suffix $\bw_1=[n, n-2, \dots, k+1, n-1]$.  To see that this image is actually in $\coll{n}{\bw_1}$, we need to check that changing the last letter of the suffix from $k$ to $n-1$ does not violate homogeneity.
   In turn, we need only to consider the case when $\sigma(\bw)$ has $[n-1, n-2, \hdots, i_{n-1}, n, n-2, \hdots, k+1, n-1]$ as a right factor. Clearly the neighbor $n-2$ appears twice between the two occurrences of $n-1$ and $\sigma(\bw) \in \coll{n}{\bw_1}$ in this case.

  If $r=n-1$, we take $m\geq k$ to be the smallest letter such that the string $[m, m+1, \hdots, n-1]$ is a right factor of $\pw$.  Then we have $\bw=s_{1i_1}\cdots s_{m-1i_{m-1}} s_ms_{m+1} \cdots s_{n-1} \bw_2$,  and we define
   \[
  \sigma(\bw) = s_{1i_1}\cdots s_{m-1i_{m-1}}s_{mk}\bw_1.
  \]
In other words, the map $\sigma$ replaces the factor $s_ms_{m+1} \cdots s_{n-1}=[m,m+1,\hdots, n-1]$ with the factor $s_{mk} = [m, m-1, \hdots, k]$ in addition to changing the suffix from $\bw_2$ to $\bw_1$.  
  
  We show now that the image
$\sigma(\bw) = s_{1i_1}\cdots s_{m-1i_{m-1}}s_{mk}\bw_1$
   is still a homogeneous word.  
  It is not hard to check that the right factor $s_{mk}\bw_1$ is homogeneous, but suppose that some letter $t$ appears in both the segment $s_{mk}$ and in some segment $s_{ji_j}$, $j \le m-1$.  Then we have $k \le t \le j \le m-1 \le n-2$. Since $t<m$, the letter $t$ does not appear in the ascending string $[m,\hdots, n-1]$ in the word $\bw$, but since $k \leq t \le n-2$, it does appear in the suffix, $[n,\hdots, t+1, t, \hdots, k]$, along with one neighbor $t+1$.  Since $\bw$ is homogeneous, there must have been another neighbor in the prefix $\pw$ which is not touched by the action of $\sigma$. Thus $\sigma(\bw)$ is also a homogeneous word, and we have shown that $\sigma(\coll{n}{\bw_2})\subset \coll{n}{\bw_1}$.  
  
  To summarize, we have
  \[
    \sigma(\bw)=
    \left\{\begin{array}{rl}
      s_{1i_1}\cdots s_{m-1i_{m-1}}s_{mk}\bw_1& \textrm{ if } \pw \textrm{ ends with }
      [m, m+1, \hdots, n-1], \\~
      \pw\bw_1 & \textrm{ otherwise}, 
    \end{array}\right.
  \]
  where $m\geq k$. Note that $\bw_0$ can not end in $n-1$ in the second case.
    For example, when $n=5$ and $k=2$, we have
  \begin{center}
  \begin{tabular}{|c|c|}
  \hline
  $\bw \in \coll{5}{[5,3,2]}$ & $\sigma(\bw) \in \coll{5}{[5,3,4]}$\\ \hline\hline
  $[3,2,1,5,3,2]$ & $[3,2,1,5,3,4]$\\ \hline
  $[4,3,2,1,5,3,2]$ & $[4,3,2,1,5,3,4]$\\ \hline
  $[1,2,3,4,5,3,2]$ & $[1,2,5,3,4]$\\ \hline
  $[2,1,4,5,3,2]$ & $[2,1,4,3,2,5,3,4]$\\ \hline
  \end{tabular}
  \end{center}

  Next, we define a map that goes in the other direction, $\tau:\coll{n}{\bw_1} \to  \coll{n}{\bw_2}$.  Suppose that $\bw = \pw\bw_1=\pw[n, n-2, \hdots, k+1, n-1] \in \coll{n}{\bw_1}$,  
and let $r$ be the last letter of $\pw$. Then by the condition of homogeneity, we must have $1\le r \le k$.

 If $r<k$, then we define $\tau(\pw \bw_1)=\pw \bw_2$, i.e. $\tau$ simply replaces the suffix $[n,\hdots, k+1, n-1]$ with the suffix $[n,\hdots, k+1, k]$.  To see that this results in a homogeneous word, we need to check only that the letter $k$ at the end of the suffix does not violate the homogeneity condition. Assume that another $k$ appears in $\pw$, and consider the last non-empty segment $s_{mr}$  of $\pw$.  Then we have $r <k \le m$.
Since there are two neighbors ($k+1$ and $k-1$) between the two occurrences of $k$, homogeneity is preserved.

  If $r=k$, then the final non-empty segment of the prefix is $s_{mk}$ for some $m$ with $k \leq m \leq n-1$.  We define 
\[\tau(\bw) =\tau( s_{1i_1}\cdots s_{m-1i_{m-1}}s_{mk}\bw_1)= s_{1i_1}\cdots s_{m-1i_{m-1}}s_{m}s_{m+1} \cdots s_{n-1} \bw_2. \]
That is,  $\tau$ replaces $s_{mk}$ with the ascending string $[m, m+1, \cdots, n-1]$ and 
the suffix $\bw_1$ with $\bw_2$.  It remains to see that $\tau(\bw)$ is in fact a homogeneous word. Notice that the left factor  $s_{1i_1}\cdots s_{m-1i_{m-1}}$ and the ascending string $[m,\hdots, n-1]$ have no letters in common, so there is nothing here to check.  Also the right factor $[m,\hdots, n-1, n, n-2, \hdots, k]$ is easily checked to be homogeneous.
  Suppose that some letter $r$ occurs in the left factor $s_{1i_1}\cdots s_{m-1i_{m-1}}$, and also in the suffix $[n,n-2, \cdots, k]$.  Since $r$ does not appear in the ascending string $[m, \hdots, n-1]$ but does appear in the suffix $[n,n-2, \hdots, k]$, it follows that $k\leq r < m$, so the letter $r$ appears in the segment $s_{mk}$ in $\pw$.  Since the word $\pw$ is homogeneous, it must be the case that two neighbors of $r$ appear between these two instances of $r$ in $\pw$.  One of them may be in the segment $s_{mk}$ which is replaced by the map $\tau$, but at least one of them must be in the left factor which remains fixed under $\tau$.  Clearly another neighbor occurs in the suffix, so the condition for homogeneity is satisfied.

  To summarize, we have, for $\bw=\pw\bw_1\in\coll{n}{\bw_1}$
   \[
    \tau(\bw)=
    \left\{\begin{array}{rl}
      s_{1i_1}\cdots s_{m-1i_{m-1}}s_{m}s_{m+1}\cdots s_{n-1} \bw_2 & \textrm{ if } \pw \textrm{ ends with }
      s_{mk}, \\~
      \pw \bw_2 & \textrm{ otherwise},
    \end{array}\right.
  \]
  where $k \leq m \leq n-1$. Note that $\bw_0$ can not end in $k$ in the second case. For example, when $n=5$ and $k=2$, we have

  \begin{center}
  \begin{tabular}{|c|c|}
  \hline
  $\bw \in \coll{5}{[5,3,4]}$ & $\tau(\bw) \in \coll{5}{[5,3,2]}$\\ \hline\hline
  $[2,1,5,3,4]$ & $[2,1,5,3,2]$\\ \hline
  $[2,5,3,4]$ & $[2,3,4,5,3,2]$\\ \hline
  $[1,4,3,2,5,3,4]$ & $[1,4,5,3,2]$\\ \hline
  \end{tabular}
  \end{center}

Now one can check that $\tau$ is both a left and a right inverse of $\sigma$, so the bijection is established.

\end{proof}

\subsection{Packets} \label{sec-pac}

The results in the previous subsection show that some collections have the same cardinality. It is natural, then, to group them together, which lead us to the following definition.

\begin{dfn}
For $0 \le k \le n$, we define the {\em $(n,k)$-packet} of collections:
\begin{itemize}
\item The $(n,0)$-packet is the set of collections labeled by suffixes of the form \[s_{n1}s_{n-1j_2}s_{nj_3}s_{n-1j_4}\cdots s_{n-1+[\ell]_2\, j_\ell}\quad (\ell \ge 2).\]
\item The $(n,k)$-packet, $1\le k \le n-3$, is the set of collections labeled by $s_{nk}$ or suffixes of the form $s_{n\, k+1}s_{n-1j_2}s_{nj_3}s_{n-1j_4}\cdots s_{n-1+[\ell]_2\, j_\ell}\quad (\ell \ge 2)$.
\item The $(n,n-2)$-packet contains only the collection labeled by $s_{n\, n-2}=[n,n-2]$.
\item The $(n,n-1)$-packet contains only the collection labeled by $s_n=[n]$.
\item The $(n,n)$-packet contains only the collection labeled by the empty suffix $[~]$. 
\end{itemize}
We will denote the $(n,k)$-packet by $\PP(n,k)$. 
As an example, Table~\ref{D4 packets} shows all of the packets in the case of $D_4$.
\end{dfn}

We record an important property of a packet:
\begin{cor} \label{num}
The collections in a packet have the same number of elements.
\end{cor}

\begin{proof}
The assertion follows from Lemma \ref{type2} and Proposition \ref{type size}.
\end{proof}

\begin{table}
 \pgfdeclarelayer{background}
\pgfdeclarelayer{foreground}
\pgfsetlayers{background,main,foreground}
\tikzstyle{coll} = [draw, very thick, fill=white,
    rectangle, rounded corners, inner sep=18pt, inner ysep=15pt, minimum height=2in]
\tikzstyle{fancytitle} =[fill=black, text=white]

\[
\begin{tikzpicture}[font=\footnotesize]

\node[coll,minimum height=.6cm, inner sep = 2mm] (C4213){
\begin{minipage}{.2\textwidth}\begin{center}
\underline{$\mathbf c^{4}_{[4,2,1,3]}$}\\[2mm]
 $[4,2,1,3]$
 \end{center}
 \end{minipage}
};

\node[coll, minimum height=.6cm, right of=C4213, xshift=3cm, inner sep = 2mm] (C42132){
\begin{minipage}{.2\textwidth}\begin{center}
\underline{$\mathbf c^{4}_{[4,2,1,3,2]}$}\\[2mm]
 $[4,2,1,3,2]$
 \end{center}
 \end{minipage}
};

\node[coll, minimum height=.6cm, right of=C42132, xshift=3cm, inner sep = 2mm] (C421324){
\begin{minipage}{.2\textwidth}\begin{center}
\underline{$\mathbf c^{4}_{[4,2,1,3,2,4]}$}\\[2mm]
 $[4,2,1,3,2,4]$
 \end{center}
 \end{minipage}
};

\node[coll, below of=C4213, yshift=-2.2cm, xshift=.85cm, minimum height=1cm, inner sep = 3mm] (C421){
\begin{minipage}{.3\textwidth}\begin{center}
\underline{$\mathbf c^{4}_{[4,2,1]}$}\\[2mm]
\begin{multicols}{2}
   $[4,2,1]$\\[2pt]
   $[3,4,2,1]$\\[2pt] \columnbreak
   $[2,3,4,2,1]$\\[2pt]
   $[1,2,3,4,2,1]$\\[2pt]
   \end{multicols}
 \end{center}
 \end{minipage}
};

\node[coll, right of=C421, xshift=.35\textwidth, minimum height=1cm, inner sep=3mm] (C423){
\begin{minipage}{.3\textwidth}\begin{center}
\underline{$\mathbf c^{4}_{[4,2,3]}$}\\[2mm]
\begin{multicols}{2}
   $[3,2,1,4,2,3]$\\[2pt]
   $[2,1,4,2,3]$\\[2pt]
   $[1,4,2,3]$\\[2pt]
   $[4,2,3]$\\[2pt]
 \end{multicols}
 \end{center}
 \end{minipage}
};

\node[coll, below of=C42132, yshift=-6cm, 
minimum height=1cm, inner sep=3mm] (C42){
\begin{minipage}{.65\textwidth}\begin{center}
\underline{$\mathbf c^{4}_{[4,2]}$}\\[2mm]
\begin{multicols}{3}
   $[4,2]$\\[2pt]
   $[3,2,1,4,2]$\\[2pt]
   $[2,1,4,2]$\\[2pt]
   $[1,4,2]$\\[2pt]
   $[2,1,3,4,2]$\\[2pt]
   $[1,3,4,2]$\\[2pt]
   $[3,4,2]$\\[2pt]
   $[2,3,4,2]$\\[2pt]
   $[1,2,3,4,2]$\\[2pt]
   \end{multicols}
 \end{center}
 \end{minipage}
};

\node[coll, below of=C42, yshift=-3.3cm, 
minimum height=1cm, inner sep=3mm] (C4){
\begin{minipage}{.75\textwidth}\begin{center}
\underline{$\mathbf c^{4}_{[4]}$}\\[2mm]
\begin{multicols}{4}
   $[1,2,3,4]$\\[2pt]
   $[1,2,4]$\\[2pt]
   $[1,3,2,4]$\\[2pt]
   $[1,3,4]$\\[2pt]
   $[1,4]$\\[2pt]
   $[2,1,3,2,4]$\\[2pt]
   $[2,1,3,4]$\\[2pt]
   $[2,1,4]$\\[2pt]
   $[2,3,4]$\\[2pt]
   $[2,4]$\\[2pt]
   $[3,2,1,4]$\\[2pt]
   $[3,2,4]$\\[2pt]
   $[3,4]$\\[2pt]
   $[4]$\\[2pt]
   \end{multicols}
 \end{center}
 \end{minipage}
};

\node[coll, below of=C4, yshift=-3.55cm, 
minimum height=1cm, inner sep=3mm] (C){
\begin{minipage}{.75\textwidth}\begin{center}
\underline{$\mathbf c^{4}_{[~]}$}\\[2mm]
\begin{multicols}{4}
   $[1,2,3]$\\[2pt]
   $[1,2]$\\[2pt]
   $[1,3,2]$\\[2pt]
   $[1,3]$\\[2pt]
   $[1]$\\[2pt]
   $[2,1,3,2]$\\[2pt]
   $[2,1,3]$\\[2pt]
   $[2,1]$\\[2pt]
   $[2,3]$\\[2pt]
   $[2]$\\[2pt]
   $[3,2,1]$\\[2pt]
   $[3,2]$\\[2pt]
   $[3]$\\[2pt]
   $[~]$\\[2pt]
   \end{multicols}
 \end{center}
 \end{minipage}
};

\begin{pgfonlayer}{background}

\path (C4213.north -| C.west)+(-.5cm,.4cm) node (ul1) {};
\path (C421324.south -| C.east)+(.5cm,-.3cm) node (lr1) {};
\draw[very thick, rounded corners, fill=black!20] (ul1) rectangle (lr1);
\node[fancytitle, xshift=.5cm,  anchor=west] at (ul1) (P1title) {$\mathcal P(4,0)$};

\path (C421.north -| C.west)+(-.5cm,.4cm)node (ul2) {};
\path (C423.south -| C.east)+(.5cm,-.3cm) node (lr2) {};
\draw[very thick, rounded corners, fill=black!20] (ul2) rectangle (lr2);
\node[fancytitle, xshift=.5cm,  anchor=west] at (ul2) (P2title) {$\mathcal P(4,1)$};

\path (C42.north -| C.west)+(-.5cm,.4cm)node (ul3) {};
\path (C42.south -| C.east)+(.5cm,-.3cm) node (lr3) {};
\draw[very thick, rounded corners, fill=black!20] (ul3) rectangle (lr3);
\node[fancytitle, xshift=.5cm,  anchor=west] at (ul3) (P3title) {$\mathcal P(4,2)$};

\path (C4.north -| C.west)+(-.5cm,.4cm)node (ul4) {};
\path (C4.south -| C.east)+(.5cm,-.3cm) node (lr4) {};
\draw[very thick, rounded corners, fill=black!20] (ul4) rectangle (lr4);
\node[fancytitle, xshift=.5cm,  anchor=west] at (ul4) (P4title) {$\mathcal P(4,3)$};

\path (C.north west)+(-.5cm,.4cm)node (ul5) {};
\path (C.south east)+(.5cm,-.3cm) node (lr5) {};
\draw[very thick, rounded corners, fill=black!20] (ul5) rectangle (lr5);
\node[fancytitle, xshift=.5cm,  anchor=west] at (ul5) (P5title) {$\mathcal P(4,4)$};

\end{pgfonlayer}
\end{tikzpicture}
\]
 \caption{The packets of $D_4$}\label{D4 packets}
\end{table}

We count the number of collections in a packet and obtain:
\begin{prop}
The size of the packet $\PP(n,k)$ is
\[
   \left | \PP(n,k) \right| =
   \left\{\begin{array}{cl}
   2^{n-2}-1 & \textrm{ if } k = 0, \\
   2^{n-k-2} & \textrm{ if } 1\le k \le n-2, \\
   1 & \textrm{ if } k= n-1, n. \\
   \end{array}
   \right.
\]
\end{prop}

It may be convenient to visualize these values in an array as in Table~\ref{Packet Size}, where the row is given by $n$ (starting at $0$), and the column is given by $k$ (also beginning at $0$).

\begin{table}[h]
\[\begin{array}{ccccccccc}
1\\
1  & 1\\
1  & 1  & 1\\
1  & 1  & 1 & 1\\
3  & 2  & 1 & 1 & 1\\
7  & 4  & 2 & 1 & 1 & 1\\
15& 8  & 4 & 2 & 1 & 1 & 1\\
31&16 & 8 & 4 & 2 & 1 & 1 & 1\\
\vdots  & \vdots &\vdots &\vdots &\vdots &\vdots &\vdots & \vdots& \ddots\\

\end{array}
\]
\caption{Triangle of Packet Sizes}\label{Packet Size}
\end{table}

\begin{proof} 
If $k=0$, a collection in the $(n,0)$-packet is determined by a sequence  $1<j_2< \dots <j_\ell \le n-1$, $\ell \ge2$. The number of such sequences is the same as the number of non-empty subsets of $\{ 2, 3, \dots , n-1 \}$, which is $2^{n-2}-1$. For $1 \le k \le n-3$,  a collection in the $(n,k)$-packet is labeled by $s_{nk}$ or is determined by a sequence  $k+1<j_2< \dots <j_\ell \le n-1$, $\ell \ge2$. The number of such sequences is the same as the number of non-empty subsets of $\{ k+2, k+3, \dots , n-1 \}$, which is $2^{n-k-2}-1$. Hence the total number of collections in the $(n,k)$-packet is $2^{n-k-2}$. The remaining cases $k=n-2, n-1, n$ are obvious.
\end{proof}

\begin{remark}
Recall that there are $2^{n-1}$ suffixes in total, and as a check, we see that 
\[ \sum_{k=0}^n  \left | \PP(n,k) \right|  =  2^{n-1} .\] 
\end{remark}

\subsection{Catalan's Triangle}
In this subsection, we will compute the size of a collection in a given packet, allowing us to classify and enumerate all homogeneous representations.  We begin by presenting a seemingly unrelated sequence.

  The array shown in Table \ref{CT} is known as {\em Catalan's Triangle} \cite{OEIS}.  The entry in the $n^{\textrm{th}}$ row and $k^\textrm{th}$ column is denoted by $C(n,k)$, for $0 \leq k \leq n$.  It can be built recursively: set the first entry $C(0,0)=1$, and then each subsequent entry is the sum of the entry above it and the entry to the left.  All entries outside of the range $0\leq k \leq n$ are considered to be $0$.
\begin{table}[h]
\[\begin{array}{ccccccccc}
1\\
1 & 1\\
1 & 2 & 2\\
1 & 3 & 5 & 5\\
1 & 4 & 9 & 14 & 14\\
1 & 5 & 14 & 28 & 42 & 42\\
1 & 6 & 20 & 48 & 90 & 132 & 132\\
1 & 7 & 27 & 75 & 165 & 297 & 429 & 429\\
\vdots & \vdots &\vdots &\vdots &\vdots &\vdots &\vdots & \vdots& \ddots\\

\end{array}
\]
\caption{Catalan's Triangle} \label{CT}
\end{table}

More precisely, for $n\geq0$ and $0\leq k \leq n$, we define
  \begin{equation}\label{sum rule}
    C(n,k) =
    \left\{
    \begin{array}{cl}
        1 & \textrm{ if } n=0 ;\\
        C(n,k-1)+C(n-1,k) & \textrm{ if } 0<k<n; \\
        C(n-1, 0) & \textrm{ if } k=0; \\
        C(n, n-1) & \textrm{ if } k=n.\\
    \end{array}
    \right.
  \end{equation}
The closed form for entries in this triangle is well known \cite{OEIS}.  For $n \geq 0$ and $0 \leq k \leq n$, we have
\begin{equation} \label{cat form}
  C(n,k) = \frac{(n+k)!(n-k+1)}{k!(n+1)!}
\end{equation}
One can also see that
\begin{equation} \label{ccc}
C_n=C(n, n-1)=C(n,n) ,
\end{equation}
where $C_n$ is the $n^{\textrm{th}}$ Catalan number.

\begin{lemma} \label{diag}
  For $n\geq 4$, each collection in the packets $\PP(n,n-1)$ and $\PP(n,n)$  contains $C_n$  homogeneous words in $\mathcal W_n$. 
\end{lemma}

\begin{proof}
  Recall that each of these packets contains only one collection: $\PP(n,n-1) = \{\coll{n}{[n]}\}$ and $\PP(n,n) = \{\coll{n}{[~]}\}$.  Since neither of the suffixes labeling these packets contains any of the letters $1,2,\hdots, n-1$, there are no restrictions on the homogeneous prefixes that can be included.  Therefore, any  homogeneous word of type $A_{n-1}$ can be a prefix. Since a fully commutative element is represented by a unique homogeneous word in canonical reduced form,  there are exactly $C_n$ of homogeneous words of type $A_{n-1}$ by Proposition~\ref{full comm count}.
\end{proof}

The previous lemma shows that the diagonal and subdiagonal of Catalan's triangle do in fact count the sizes of collections in the corresponding packets.  We now give the main theorem of this section, which says that entries in the rest of Catalan's triangle also count the sizes of collections, naturally extending Lemma \ref{diag}. This theorem allows us to classify the homogeneous representations.

\begin{thm}\label {main D}Assume that $n\ge 4$ and $0\le k \le n$. Then any collection in the packet $\PP(n,k)$ contains exactly $C(n,k)$
  elements.
\end{thm}

\begin{proof}
  Note that Lemma~\ref{diag} proves the case when $k=n-1 \textrm{ or } n$.  We next consider the packet $\PP(n,0)$, which consists of the collections labeled by the prefixes  $s_{n1}s_{n-1j_2}s_{nj_3}s_{n-1j_4}\cdots s_{n-1+[\ell]_2\, j_\ell}$ $(\ell \ge 2)$. By Lemma \ref{type2}, it is enough to consider the collection $\bf{c}$ labeled by $s_{n1}s_{n-1}$. 
We claim that no prefix is possible except the empty word  $[~]$.  Indeed, if a word $\bw$ in $\bf{c}$ contains a non-empty prefix $\pw$ ending with the letter $r$ for $1\leq r \leq n-1$, then the word $\bw$ contains the factor $[ r, n, n-2, \dots, 1, n-1]$ and we immediately reach a contradiction to homogeneity.  Thus the collection $\bf{c}$, and hence every collection in $\PP(n,0)$, contains only one element, the suffix itself. Since  $C(n,0)=1$, we are done with this case.

We see in Table \ref{D4 packets} that any collection in the packet $\PP(4,k)$ contains exactly $C(4,k)$ elements. Thus the assertion of the theorem is true for $n=4$. 
Recall the recursive definition of $C(n,k)$:   for $1 \le k \le n-1$,
\begin{equation} \label{sum-rule}
C(n,k) = C(n, k-1) + C(n-1, k). 
\end{equation}
We will prove that the collection sizes of the packets satisfy the relation given in equation \eqref{sum-rule}. Since we already checked the base cases ($k=0$, $n=4$), the theorem will be established by induction. Further, by Corollary \ref{num}, it will be enough to check that the relation holds for a single collection in each of the packets.

Now assume that $n > 4$ and $1 \leq k \leq n-2$.  We define  
\begin{align*} & \bw_1= s_{nk}s_{n-1}=[n, n-2, \dots, k, n-1],  \quad \bw_2=s_{nk}=[n, n-2, \dots , k], \\
&  \bw_3=\begin{cases} [n-1, n-3, \dots , k] &\quad \text{ if } k<n-2 , \\ [n-1] & \quad \text{ if } k=n-2 . \end{cases}
\end{align*}

 Then $\coll{n}{\bw_1} \in \PP(n, k-1)$, $\coll{n}{\bw_2} \in \PP(n, k)$ and $\coll{n-1}{\bw_3} \in \PP(n-1, k)$. We will give an explicit bijection from  $\coll{n}{\bw_1}\cup\coll{n-1}{\bw_3}$ to $\coll{n}{\bw_2}$. Since $\coll{n}{\bw_1}\cap\coll{n-1}{\bw_3} = \emptyset$ as sets of formal words, we will have
 \begin{equation} \label{end}  |\coll{n}{\bw_1}|+|\coll{n-1}{\bw_3}| = |\coll{n}{\bw_2}| ,\end{equation}
 and the proof will be completed by induction since the collection sizes satisfy the same recursive relation as \eqref{sum-rule}.

   Define the map $\varphi_1: \coll{n}{\bw_1} \to \coll{n}{\bw_2}$ by
   \[
    \varphi_1(\pw \bw_1) = \varphi(\pw [n, n-2, \dots , k, n-1] )= \pw \bw_2 = \pw [n, n-2, \dots , k] .
   \]  Clearly, removing the last letter of the suffix will not affect the homogeneity of a word, and hence $\varphi_1$  maps $\coll{n}{\bw_1}$ into $\coll{n}{\bw_2}$.
   Similarly,  we define the map $\varphi_2: \coll{n-1}{\bw_3} \to  \coll{n}{\bw_2}$ to be
   \[\varphi_2(\pw\bw_3) = \pw s_{n-1} \bw_2= \pw s_{n-1} [ n, n-2, n-3, \dots, k].\]
To see that the image is homogeneous, we first consider the case when  $\pw$ contains the letter $n-2$. In such a case, there are two neighbors $n-1$ and $n$ between two occurrences of $n-2$. Now consider the other case when $\bw_0$ contains $i<n-2$. In this case, there is a neighbor to
the right of $i$ in $\pw$ by homogeneity of $\bw_0\bw_3$ and also an $i+1$ in the suffix to
the left of $i$.
Hence,   $\varphi_2$  maps $\coll{n-1}{\bw_3}$ into $\coll{n}{\bw_2}$.
 Note that the images of $\varphi_1$ and $\varphi_2$ are disjoint: the words in the image of $\varphi_2$ all have prefixes that end with the letter $n-1$, while none of the words in the image of $\varphi_1$ do since $\pw$ cannot end with $n-1$ due to the homogeneity condition on $\bw_0\bw_1$.

Finally, we define the map $\varphi: \coll{n}{\bw_1}\cup\coll{n-1}{\bw_3} \to \coll{n}{\bw_2}$ to be the combination of $\phi_1$ and $\phi_2$, i.e. the restriction of $\phi$  to $\coll{n}{\bw_1}$ is defined to be $\varphi_1$ and the restriction of $\phi$ to $\coll{n-1}{\bw_3}$ is defined to be $\varphi_2$.
  Figure~\ref{Packet Map} shows an example of the maps $\varphi_1$ and $\varphi_2$ in the case of $n=5$ and $k=2$.
   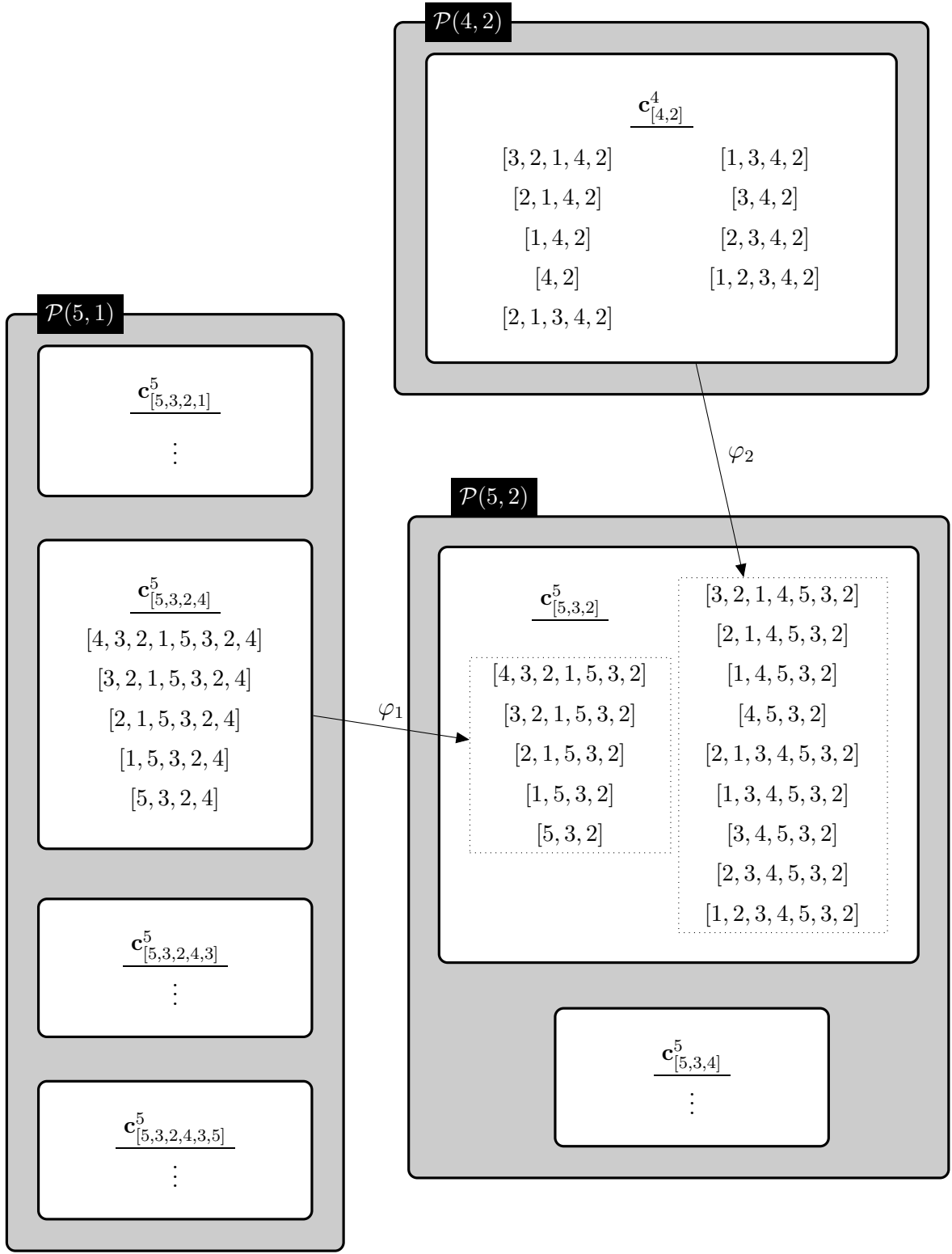
\begin{figure}

\pgfdeclarelayer{background}
\pgfdeclarelayer{foreground}
\pgfsetlayers{background,main,foreground}

\tikzstyle{coll} = [draw, very thick, fill=white,
    rectangle, rounded corners, inner sep=18pt, inner ysep=15pt, minimum height=2in]
\tikzstyle{fancytitle} =[fill=black, text=white]

\[
\begin{tikzpicture}

\node[coll] (C5324){
\begin{minipage}{.2\textwidth}  \begin{center}
 {\underline{~$\coll{5}{[5,3,2,4]}$~}\\[2mm]
    $[4,3,2,1,5,3,2,4]$\\[2mm]
    $[3,2,1,5,3,2,4]$\\[2mm]
    $[2,1,5,3,2,4]$\\[2mm]
    $[1,5,3,2,4]$\\[2mm]
    $[5,3,2,4]$\\[2mm]}
    \end{center}
\end{minipage}
};

\node[coll, above of = C5324, minimum height=1cm, yshift=3.5cm] (C5321){
\begin{minipage}{.2\textwidth}  \begin{center}
 \underline{~$\mathbf c^{5}_{[5,3,2,1]}$~}\\[2mm]
$\vdots$
\end{center}
\end{minipage}
};

\node[coll, below of = C5324, minimum height=1cm, yshift=-3.5cm] (C53243){
\begin{minipage}{.2\textwidth} \begin{center} 
 \underline{~$\mathbf c^{5}_{[5,3,2,4,3]}$~}\\
   $\vdots$\\
   \end{center}
\end{minipage}
};

\node[coll, below of = C53243, minimum height=1cm, yshift=-2cm] (C532435){
\begin{minipage}{.2\textwidth}  \begin{center}
    \underline{~$\mathbf c^{5}_{[5,3,2,4,3,5]}$~}\\
   $\vdots$\end{center}
\end{minipage}
};

\node[coll,
right of=C5324,
yshift=8cm,
node distance=8cm,
anchor = center] (C42){
\begin{minipage}{0.4 \textwidth} \begin{center}
    \underline{~$\mathbf c^{4}_{[4,2]}$~}\\[2pt]
  \begin{multicols}{2}
    $[3,2,1,4,2]$\\[2mm]
    $[2,1,4,2]$\\[2mm]
    $[1,4,2]$\\[2mm]
    $[4,2]$\\[2mm]
    $[2,1,3,4,2]$\\[2mm]
    \columnbreak
    $[1,3,4,2]$\\[2mm]
    $[3,4,2]$\\[2mm]
    $[2,3,4,2]$\\[2mm]
    $[1,2,3,4,2]$\\ 
\end{multicols}\end{center}
\end{minipage}
};

\begin{pgfonlayer}{foreground}
\node[draw, dotted, right of=C5324,node distance=6.5cm, yshift=-1cm,inner sep=1pt, inner ysep=1mm](C532A){
\begin{minipage}{.2\textwidth}\begin{center}
    $[4,3,2,1,5,3,2]$\\[2mm]
    $[3,2,1,5,3,2]$\\[2mm]
    $[2,1,5,3,2]$\\[2mm]
    $[1,5,3,2]$\\[2mm]
    $[5,3,2]$\end{center}
    \end{minipage}
};

\node[draw, dotted,
 right of=C532A,
 node distance = 3.5cm, inner sep=1mm] (C532B){
\begin{minipage}{0.2 \textwidth}\begin{center}
    $[3,2,1,4,5,3,2]$\\[2mm]
    $[2,1,4,5,3,2]$\\[2mm]
    $[1,4,5,3,2]$\\[2mm]
    $[4,5,3,2]$\\[2mm]
    $[2,1,3,4,5,3,2]$\\[2mm]
    $[1,3,4,5,3,2]$\\[2mm]
    $[3,4,5,3,2]$\\[2mm]
    $[2,3,4,5,3,2]$\\[2mm]
    $[1,2,3,4,5,3,2]$
  \end{center}
\end{minipage}
};

\node[above of=C532A, yshift=1.5cm] (C532) {
  \underline{~$\mathbf c^{5}_{[5,3,2]}$~}
};

\end{pgfonlayer}

\node[coll, below of = C532B, yshift=-4.3cm, xshift=-1.5cm, minimum height=1cm] (C534){
  \begin{minipage}{.2\textwidth}  \begin{center}
    \underline{~$\mathbf c^{5}_{[5,3,4]}$~}\\
   $\vdots$\end{center}
\end{minipage}
};

\path (C532B.north east)+(.5cm,.5cm) node (ur){};
\draw[very thick, rounded corners, fill=white] (C532A.south west)+(-.5cm,-1.8cm) rectangle (ur);

\path (C5324) edge[->, >=triangle 45] node[above] {\large$\varphi_1$} (C532A)
	(C42) edge[->, >=triangle 45] node[above right] {\large$\varphi_2$} (C532B);

\begin{pgfonlayer}{background}
\path (C5321.north east)+(.5cm,.5cm) node (urP1) {};
\draw[very thick, rounded corners, fill=black!20] (C532435.south west)+(-.5cm,-.5cm) rectangle (urP1);
\node[fancytitle, yshift=.5cm, anchor=west] at (C5321.north west)(P1title) {$\mathcal P(5,1)$};

\path (C42.north east)+(.5cm,.5cm) node (urP2) {};
\draw[very thick, rounded corners, fill=black!20] (C42.south west)+(-.5cm,-.5cm) rectangle (urP2);
\node[fancytitle, yshift=.5cm, anchor=west] at (C42.north west)(P1title) {$\mathcal P(4,2)$};

\path (C534.south -| C532A.west)+(-1cm,-.5cm) node (ll) {};
\path (ur)+(.5cm,.5cm) node (urP3) {};
\draw[very thick, rounded corners, fill=black!20] (ll) rectangle (urP3);
\node[fancytitle, xshift=-1.2cm, yshift=1.3cm, anchor=west] at (C532.north west)(P1title) {$\mathcal P(5,2)$};
\end{pgfonlayer}	

\end{tikzpicture}
\]
    \caption{The maps $\varphi_1$ and $\varphi_2$ into the packet $\PP(5,2)$.}\label{Packet Map}
   \end{figure}

 Now we define the  map $\rho: \coll{n}{\bw_2} \to \coll{n}{\bw_1}\cup\coll{n-1}{\bw_3}$ to be given by the rule:
   \[
   \rho(\pw\bw_2) = \rho(\pw [n, n-2, \dots , k]) =
   \left\{\begin{array}{ll}
    \pw [ n-3, \dots, k] \in \coll{n-1}{\bw_3}  & \textrm{ if } \pw \textrm{ ends with } n-1, \\~
   \pw \bw_1 \in \coll{n}{\bw_1} & \textrm{ otherwise.}
   \end{array}\right.
   \]
In the case where $\pw$ ends with $n-1$, we first check homogeneity with the letter $n-3$ in $\pw$. In passing to $D_{n-1}$, the letters $n-1$ and $n-3$ become neighbors and homogeneity follows because there is another neighbor of $n-3$ further to the right in $\pw$. Now we check homogeneity when $\pw$ contains $i<n-3$. Note that there is a neighbor to
the right of $i$ in $\pw$ by homogeneity of $\bw_0\bw_2$ and also an $i+1$ in the suffix to
the left of $i$.
 In the case where $\pw$ does not end with $n-1$, it is still  possible that $\pw$ contains $n-1$, but since it would be part of a descending sequence at the end of the prefix, there will always be two instances of the neighbor $n-2$ between the two occurrences of $n-1$.  Thus the map $\rho$ is well defined.

Now one can check that $\rho$ is the two-sided inverse of $\phi$. In particular, if we restrict $\rho$ to the words whose prefixes end with $n-1$, then we obtain the inverse for $\varphi_2$, while if we restrict to the prefixes not ending in $n-1$, we have the inverse for $\varphi_1$. 
This establishes \eqref{end} and completes the proof.	
\end{proof}

We obtain immediate consequences for homogeneous representations of $KLR$ algebras.

\begin{cor} \label{cor-DD}
Assume that  $R=\bigoplus_{\alpha\in Q_+}R_\alpha$ is the  type-$D_{n}$ KLR algebra.
\begin{enumerate}
\item The set of irreducible homogeneous representations of $R$ is decomposed into packets and collections according to the decomposition of the set of fully commutative elements (or of the set of homogeneous words in $\mathcal W_n$).
\item Each entry in Catalan's triangle counts the number of homogeneous representations of $R$ in the same collection. 
\end{enumerate}
\end{cor}

\begin{cor} \label{cor-end}
For $n\ge 4$, we obtain the identity:
\begin{equation} \label{eqn-end} \sum_{k=0}^n C(n,k) \left | \PP(n, k) \right | = \frac {n+3} 2 C_n -1 .\end{equation}
\end{cor}

\begin{proof}
The identity follows from Proposition \ref{full comm count} and Theorem \ref{main D}. 
\end{proof}

\begin{remark}
As pointed out in Introduction,  the identity \eqref{eqn-end} suggests that there might be a  representation theoretic construction, in which   $C(n,k)$ would correspond to the dimension of a representation and  $|\PP(n,k)|$ to its multiplicity. Also recall that the right-hand side of \eqref{eqn-end} is equal to the dimension of the Temperley--Lieb algebra of type $D_n$. 
\end{remark}

\vskip 1 cm

\bibliography{Hom_Rep}
\bibliographystyle{amsalpha}

\end{document}